\documentclass[12pt,letterpaper]{amsart}
\usepackage[letterpaper,hmargin=1.25in,vmargin=1.5in]{geometry}
\usepackage{latexsym, amssymb, amsmath}

\def\R{\mathbb R}

\newtheorem{theorem}{Theorem}[section]
\newtheorem{lemma}[theorem]{Lemma}
\newtheorem{corollary}[theorem]{Corollary}
\newtheorem{proposition}[theorem]{Proposition}
\newtheorem{remark}[theorem]{Remark}
\theoremstyle{definition}
\newtheorem{definition}[theorem]{Definition}

\makeatletter
    
    \@addtoreset{equation}{section}
  \makeatother

\newcommand{\Ric}{{\rm Ric}}
\newcommand{\Rm}{{\rm Rm}}

\begin{document}

\title
[Complete gradient Einstein-type Sasakian manifolds with $\alpha=0$]
{Complete gradient Einstein-type Sasakian manifolds with $\alpha=0$}

\author{Shun Maeta}
\address{Department of Mathematics, Chiba University, 1-33, Yayoicho, Inage, Chiba, 263-8522, Japan.}
\curraddr{}
\email{shun.maeta@faculty.gs.chiba-u.jp~{\em or}~shun.maeta@gmail.com}
\thanks{The author is partially supported by the Grant-in-Aid for Scientific Research (C), No.23K03107, Japan Society for the Promotion of Science.}

\date{}

\dedicatory{}


\commby{}

\begin{abstract}
Catino, Mastrolia, Monticelli, and Rigoli have launched an ambitious program to study known geometric solitons from a unified perspective, which they term Einstein-type manifolds. This framework allows one to treat Ricci solitons, Yamabe solitons, and all of their generalizations simultaneously. Einstein-type manifolds are characterized by four constants $\alpha, \beta, \mu$ and $\rho$.
In this paper, we show that when $\alpha = 0$, complete gradient Einstein-type Sasakian manifolds are trivial or isometric to the unit sphere. As a consequence, many geometric solitons on Sasakian manifolds turn out to be trivial or isometric to the unit sphere.

\end{abstract}

\maketitle

\bibliographystyle{amsalpha}


\section{Introduction}\label{intro}

Geometric flows have become one of the most powerful tools for understanding Riemannian manifolds (cf. \cite{Brendle05}, \cite{Brendle07}, \cite{BS08}, \cite{Chow92}, \cite{Perelman02}, \cite{SS03}, \cite{Ye94}). To analyze geometric flows, it is important to study their self-similar solutions, which are known as geometric solitons. 
Geometric solitons have become central to geometry (see \cite{BR13}, \cite{BGV16}, \cite{CSZ12}, \cite{CC12}, \cite{CC13}, \cite{CCCMM14}, \cite{CD08}, \cite{CMM12}, \cite{CMMR17}, \cite{CMM16}, \cite{CMM17}, \cite{Chowetal07}, \cite{CLN06}, \cite{DS13}, \cite{FOW09}, \cite{Hamilton89}, \cite{HL14}, \cite{Maeta21}, \cite{Maeta23}, \cite{Maeta24}, \cite{PBS23}, \cite{Perelman02}). 
There are various types of geometric solitons, including Ricci solitons, Yamabe solitons, $k$-Yamabe solitons, quasi-Yamabe solitons, almost Yamabe solitons, and conformal solitons. 
In particular, it has recently been recognized that quasi-type solitons are highly effective for resolving problems arising in other areas of research (cf. \cite{CMR24}; see also \cite{CL24}).

An ambitious project was initiated by Catino, Mastrolia, Monticelli, and Rigoli (cf. \cite{CMMR17}). They introduced the notion of Einstein-type manifolds, a unifying framework that encompasses previously known geometric solitons and which now appears to clarify why analogous results recur across different soliton classes.

Let $(M,g)$ be a Riemannian manifold. If there exist $X\in\mathfrak{X}(M)$ and $\lambda\in C^\infty(M)$ such that 
\begin{equation*}
\alpha\,\Ric + \frac{\beta}{2} L_X g + \mu\, X^\flat\otimes X^\flat = (\rho R + \lambda) g,
\end{equation*}
for some constants $\alpha,\beta,\mu,\rho\in\mathbb{R}$ with $(\alpha,\beta,\mu)\neq(0,0,0)$, then $(M,g,X)$ is called an {\em Einstein-type manifold}.
Here, $\Ric$ denotes the Ricci tensor of $M$, $L_X$ denotes the Lie derivative along $X$, and $R$ is the scalar curvature of $M$.
If $X=\nabla F$ for some $F\in C^\infty(M)$, then $(M,g,F)$ is called a {\em gradient} Einstein-type manifold:
\begin{equation}\label{getm}
\alpha\,\Ric + \beta\,\nabla\nabla F + \mu\, \nabla F\otimes \nabla F = (\rho R + \lambda) g,
\end{equation}
where $\nabla\nabla F$ is the Hessian of $F$, and $F$ is called the potential function.
If $F$ is constant, then $(M,g,F)$ is called {\em trivial}.

As pointed out in \cite{CMMR17}, gradient Einstein-type manifolds include 
gradient Yamabe solitons, gradient almost Yamabe solitons, gradient $k$-Yamabe solitons, gradient conformal solitons, and quasi-Yamabe gradient solitons (see, for example, \cite{Hamilton82}, \cite{BR13}, \cite{CMM12}, \cite{HL14}).

In this paper, we consider gradient Einstein-type manifolds with $\alpha=0$, which include all the classes above. 
In \cite{CMMR17}, Catino, Mastrolia, Monticelli, and Rigoli provided a structure theorem for complete $n$-dimensional $(n\ge 3)$ gradient Einstein-type manifolds with $\alpha=0$ (see Theorem~1.4 in \cite{CMMR17}). 
Theorem~1.4 of \cite{CMMR17} generalizes results in \cite{CSZ12}, \cite{CMM12}, \cite{Tashiro65}. 

As is well known, Sasakian manifolds provide the odd-dimensional counterpart of K\"ahler geometry via their K\"ahler cones and Reeb foliations. 
In particular, Sasaki-Einstein manifolds appear prominently in high-energy physics and string theory, for example as internal spaces in supersymmetric AdS/CFT compactifications, where they model the geometric structure underlying certain superconformal field theories (see, for example, \cite{BG08}, \cite{Sparks11}, \cite{GMSW05}). 

Therefore, in this paper, we consider complete gradient Einstein-type Sasakian manifolds with $\alpha=0$. 
Interestingly, we can show that they are trivial or isometric to the unit sphere.

\begin{theorem}\label{main}
Let $(M^{2n+1},\phi,\xi,\eta,g,F)$ be a complete gradient Einstein-type Sasakian manifold with $\alpha=0$.
Then exactly one of the following statements holds.
\begin{enumerate}
\item[{\rm (i)}]
$\beta=0$. The manifold $(M,g)$ is trivial.
\item[{\rm (ii)}]
$\beta\neq0$.
Either $(M,g)$ is trivial, or $(M,g)$ is isometric to the unit sphere $(I\times \mathbb{S}^{2n},dr^2+\sin^2r\,g_{\mathbb{S}^{2n}})$ where $I$ is a compact interval and $g_{\mathbb{S}^{2n}}$ is the Riemannian metric of the unit sphere $\mathbb{S}^{2n}$. 
In the nontrivial case, there exists a nonconstant smooth function $v$ on $M$ such that
\begin{equation*}
\nabla\nabla v=-vg,
\end{equation*}
and $v$ has exactly two critical points. In particular, $v(r)=a\cos r$ for some constant $a>0$.

Moreover, if $c:=-\frac{\mu}{\beta}=0$ then we can take $v=F-k$ for some constant $k$, hence
\[
F(r)=k+a\cos r.
\]
If $c\neq0$ then we can take $v=e^{-cF}-k$ for some constant $k$, hence
\[
F(r)=-\frac{1}{c}\ln(k+a\cos r)
\qquad (k>a).
\]
\end{enumerate}
\end{theorem}

We now give an outline of the proof.
From \eqref{getm} with $\alpha=0$, we rewrite the equation as
\[
\beta \nabla\nabla F + \mu\, \nabla F\otimes  \nabla F = \psi\, g
\qquad ((\beta,\mu)\neq(0,0))
\]
for some $\psi\in C^\infty(M)$.
Since the case $\beta=0$ can be proved relatively easily, we only describe the case $\beta\neq 0$.
Dividing by $\beta$ and setting $c=-\mu/\beta$, we obtain
\[
\nabla\nabla F - c\, \nabla F\otimes  \nabla F = \psi\, g.
\]
Differentiating yields a curvature identity for $\Rm(X,Y)\nabla F$, where the Riemannian curvature tensor $\Rm$ is defined by
$\Rm(X,Y)Z=\nabla _X\nabla_YZ-\nabla _Y\nabla_XZ-\nabla _{[X,Y]}Z,~~(X,Y,Z\in\frak{X}(M)).$ 
Tracing this identity yields a formula for $\Ric(\phi X,\phi\nabla F)$. On the other hand, the Sasakian condition gives another formula for $\Ric(\phi X,\phi\nabla F)$. 
Using standard Sasakian identities (in particular $\nabla_X\xi=-\phi X$), we obtain the formula
\begin{equation*}
\begin{aligned}
\nabla \psi=(c\psi-1)\nabla F.
\end{aligned}
\end{equation*}
By the formula, we obtain the Obata-type equation
\[
\nabla \nabla v=-vg.
\]
By Tashiro's theorem, we obtain that $M$ is isometric to the unit sphere.


\section{Preliminaries}
In this section, we present the definitions needed to introduce Sasakian manifolds. After defining Sasakian manifolds, we recall standard consequences and identities. For details, see \cite{YK84}, \cite{BG08}.

\begin{definition}[Almost contact manifold]
An \emph{almost contact structure} on a smooth $(2n+1)$-dimensional manifold $M$ is a triple $(\phi,\xi,\eta)$ with $\phi\in \Gamma(T^1_1 M)$, $\xi\in\mathfrak{X}(M)$, and $\eta\in\Omega^1(M)$ such that
\begin{equation*}
\eta(\xi)=1,\qquad \phi^2=-I+\eta\otimes \xi.
\end{equation*}
A manifold endowed with such a triple is called an \emph{almost contact manifold}.
\end{definition}

\begin{remark}\label{rem1}
For any almost contact manifold $(M^{2n+1},\phi,\xi,\eta)$, the following identities hold:
\[
\phi\,\xi=0,~\eta(\phi X)=0 \quad\text{for any vector field}~X.
\]
\end{remark}

\begin{definition}[Almost contact metric manifolds]
An \emph{almost contact metric structure} on $M^{2n+1}$ is a quadruple $(\phi,\xi,\eta,g)$ where $(\phi,\xi,\eta)$ is an almost contact structure and $g$ is a Riemannian metric satisfying
\begin{equation}\label{eq:acmetric}
\eta(X)=g(X,\xi),\qquad g(\phi X,\phi Y)=g(X,Y)-\eta(X)\eta(Y).
\end{equation}
A manifold endowed with such a quadruple is called an \emph{almost contact metric manifold}.
The associated \emph{fundamental $2$-form} is $\Phi(X,Y):=g(X,\phi Y)$.
\end{definition}


\begin{definition}[Contact metric manifold]
A \emph{contact metric manifold} is an almost contact metric manifold $(M^{2n+1},\phi,\xi,\eta,g)$ such that
\begin{equation*}
d\eta=\Phi.
\end{equation*}
\end{definition}

\begin{definition}[Nijenhuis tensor]
Given an almost contact manifold $(M^{2n+1},\phi,\xi,\eta)$, define an almost complex structure $J$ on $M\times \R$ by
\begin{equation*}
J\bigl(X,f\,\tfrac{\partial}{\partial t}\bigr):=\bigl(\phi X-f\,\xi,\ \eta(X)\tfrac{\partial}{\partial t}\bigr).
\end{equation*}
The \emph{Nijenhuis tensor} of $J$ is the $(1,2)$-tensor $N_J$ defined by
\begin{equation*}
N_J(U,V)=[JU,JV]-J[JU,V]-J[U,JV]+J^2[U,V],
\end{equation*}
for vector fields $U,V$. 
The structure $(\phi,\xi,\eta)$ is called \emph{normal} if $N_J\equiv0$.
\end{definition}

\begin{definition}[Sasakian manifold]
A \emph{Sasakian manifold} is a contact metric manifold $(M^{2n+1},\phi,\xi,\eta,g)$ that is normal.
\end{definition}

We will use the following known results later.

\begin{theorem}\label{thm:sasaki-char}
Let $(M^{2n+1},\phi,\xi,\eta,g)$ be an almost contact metric manifold. Then $M$ is Sasakian if and only if
\begin{equation}\label{Sasakirici}
(\nabla_X\phi)Y=g(X,Y)\,\xi-\eta(Y)\,X \quad \text{for any vector fields}~X,Y.
\end{equation}
\end{theorem}

\begin{lemma}\label{lem:ricci-sasaki}
If $(M^{2n+1},\phi,\xi,\eta,g)$ is a Sasakian manifold, then
\begin{align*}
\Ric(X,\xi)&=2n\,\eta(X),\\
\Ric(\phi X,\phi Y)&=\Ric(X,Y)-2n\,\eta(X)\eta(Y).
\end{align*}
\end{lemma}

\begin{definition}[$K$-contact]
A contact metric manifold $(M^{2n+1},\phi,\xi,\eta,g)$ is called \emph{$K$-contact} if and only if $\xi$ is a Killing vector field.
\end{definition}

\begin{proposition}\label{prop:Kcontact}
On a contact metric manifold $(M^{2n+1},\phi,\xi,\eta,g)$, the following are equivalent.
\begin{enumerate}
\item
The structure is $K$-contact.
\item
$\nabla_X\xi=-\,\phi X$ holds for any vector field $X$.
\end{enumerate}
\end{proposition}


\section{Proof of Theorem \ref{main}}

In this section, we show Theorem \ref{main}.

\begin{proof}[Proof of Theorem~\ref{main}]

Since $\lambda$ is a smooth function, the equation $\eqref{getm}$ can be reduced to
\begin{equation}\label{getm0}
\beta \nabla \nabla F+\mu \nabla F \otimes\nabla F =\psi g,
\end{equation}
for some $\psi\in C^\infty(M)$ and constants $(\beta,\mu)\not=(0,0)$.
If $\beta\not=0$, without loss of generality, \eqref{getm0} can be written as
\[
\nabla \nabla F-c\nabla F\otimes\nabla F=\psi g,
\]
for some $c\in\mathbb{R}$.
If $\beta=0$, without loss of generality, \eqref{getm0} can be written as
\[
\nabla F\otimes\nabla F=\psi g.
\]

\noindent
Case 1: $\beta\not=0$.

Since $M$ is a gradient Einstein-type manifold with $\alpha=0$, we have
\begin{equation}\label{getm1}
\nabla_X\nabla F=\psi X+cg(X,\nabla F)\nabla F.
\end{equation}
By the equation, we have
\begin{equation}\label{Rm}
\begin{aligned}
\Rm(X,Y)\nabla F
&=\nabla _X\nabla_Y\nabla F-\nabla_Y\nabla_X\nabla F-\nabla_{[X,Y]}\nabla F\\
&=\nabla _{X}(\psi Y+cg(Y,\nabla F)\nabla F)\\
&\quad-\nabla _{Y}(\psi X+cg(X,\nabla F)\nabla F)\\
&\quad -\{\psi [X,Y]+cg([X,Y],\nabla F)\nabla F\}\\
&=(X\psi)Y+\psi\nabla_XY
+cg(\nabla_XY,\nabla F)\nabla F\\
&\quad +cg(Y,\nabla_X\nabla F)\nabla F
+cg(Y,\nabla F)\nabla_X\nabla F\\
&\quad -\{
(Y\psi)X+\psi\nabla_YX
+cg(\nabla_YX,\nabla F)\nabla F\\
&\quad +cg(X,\nabla_Y\nabla F)\nabla F
+cg(X,\nabla F)\nabla_Y\nabla F
\}\\
&\quad -\psi [X,Y]-cg([X,Y],\nabla F)\nabla F\\
&=(X\psi)Y
+cg(Y,\psi X+cg(X,\nabla F)\nabla F)\nabla F\\
&\quad +cg(Y,\nabla F)\{\psi X+cg(X,\nabla F)\nabla F\}\\
&\quad 
-\{
(Y\psi)X
+cg(X,\psi Y+cg(Y,\nabla F)\nabla F)\nabla F\\
&\quad +cg(X,\nabla F)\{\psi Y+cg(Y,\nabla F)\nabla F\}
\}\\
&=(X\psi-c\psi XF)Y-(Y\psi-c\psi YF)X.
\end{aligned}
\end{equation}
By taking the trace, one has
\begin{equation}\label{Ric}
\Ric(Y,\nabla F)=g(Y,-2n(\nabla\psi-c\psi\nabla F)).
\end{equation}

Since $M$ is a Sasakian manifold, it follows from \eqref{Sasakirici} (see page 277 in \cite{YK84}) that
\begin{equation*}
\begin{aligned}
\Rm(X,Y)\phi\nabla F
&=\phi (\Rm(X,Y)\nabla F)\\
&\quad -\eta(\nabla_Y\nabla F)X+\eta(\nabla_X\nabla F)Y\\
&\quad +g(Y,\nabla F)\nabla _X\xi-(X(\eta(\nabla F)))Y\\
&\quad -\{g(X,\nabla F)\nabla _Y\xi-(Y(\eta(\nabla F)))X\}.
\end{aligned}
\end{equation*}
By taking the trace, one has
\begin{equation}\label{Ricphi}
\begin{aligned}
\Ric(Y,\phi \nabla F)
&=
g(\phi(\Rm(e_i,Y)\nabla F),e_i)\\
&\quad+g(-\eta(\nabla_Y\nabla F)e_i+\eta(\nabla_{e_i}\nabla F)Y,e_i)\\
&\quad +g(Y,\nabla F)g(\nabla _{e_i}\xi,e_i)
-(e_i(\eta(\nabla F)))g(Y,e_i)\\
&\quad-\{g(e_i,\nabla F)g(\nabla_Y\xi,e_i)
-(2n+1)Y(\eta(\nabla F))\}\\
&=
g(Y,-\phi(\nabla \psi-c\psi\nabla F))\\
&\quad -2n\{\psi\eta(Y)+cg(Y,\nabla F)\eta(\nabla F)\}\\
&\quad+g(Y,\nabla F)g(\nabla _{e_i}\xi,e_i)
-(e_i(\eta(\nabla F)))g(Y,e_i)\\
&\quad-\{g(e_i,\nabla F)g(\nabla_Y\xi,e_i)
-(2n+1)Y(\eta(\nabla F))\},
\end{aligned}
\end{equation}
where $\{e_i\}$ is an orthonormal frame and the last equality follows from \eqref{getm1} and \eqref{Rm}.
Take $Y=\phi X$. By Lemma \ref{lem:ricci-sasaki}, one has
\[
\begin{aligned}
\Ric(\phi X,\phi \nabla F)
&=\Ric(X,\nabla F)-2n\eta(X)\eta(\nabla F)\\
&=
g(X,-2n(\nabla \psi-c\psi\nabla F))-2n\eta(X)\eta(\nabla F),
\end{aligned}
\]
where the last equality follows from \eqref{Ric}.
On the other hand,
\[
\begin{aligned}
\Ric(\phi X,\phi \nabla F)
&=g(\phi X,-\phi(\nabla \psi-c\psi\nabla F))\\
&\quad -2n\{\psi\eta(\phi X)+cg(\phi X,\nabla F)\eta(\nabla F)\}\\
&\quad+g(\phi X,\nabla F)g(\nabla _{e_i}\xi,e_i)
-(e_i(\eta(\nabla F)))g(\phi X,e_i)\\
&\quad-\{g(e_i,\nabla F)g(\nabla_{\phi X}\xi,e_i)
-(2n+1)(\phi X)(\eta(\nabla F))\}\\
&=g(X,-(\nabla \psi-c\psi\nabla F))-\eta(X)\eta(-(\nabla \psi-c\psi\nabla F))\\
&\quad -2ncg(\phi X,\nabla F)\eta(\nabla F)\\
&\quad+g(\phi X,\nabla F)g(\nabla _{e_i}\xi,e_i)
-(e_i(\eta(\nabla F)))g(\phi X,e_i)\\
&\quad-\{g(e_i,\nabla F)g(\nabla_{\phi X}\xi,e_i)
-(2n+1)(\phi X)(\eta(\nabla F))\},
\end{aligned}
\]
where the last equality follows from \eqref{eq:acmetric}.
Combining the above two equations, we have
\begin{equation}\label{key}
\begin{aligned}
&g(X,-(2n-1)(\nabla \psi-c\psi\nabla F))
-2n\eta(X)\eta(\nabla F)
-\eta(X)\eta(\nabla \psi-c\psi\nabla F)\\
& -g(\phi X,\nabla F)g(\nabla _{e_i}\xi,e_i)
+(e_i(\eta(\nabla F)))g(\phi X,e_i)\\
&+g(e_i,\nabla F)g(\nabla _{\phi X}\xi,e_i)
-(2n+1)(\phi(X))(\eta(\nabla F))
+2ncg(\phi X,\nabla F)\eta(\nabla F)=0.
\end{aligned}
\end{equation}

We consider the left-hand side of \eqref{key}.

We consider the second and third terms of the left-hand side of \eqref{key}.
By \eqref{eq:acmetric}, we have
\[
-2n\eta(X)\eta(\nabla F)
-\eta(X)\eta(\nabla \psi-c\psi\nabla F)
=g(X,-2n\eta(\nabla F)\xi-\eta(\nabla \psi-c\psi\nabla F)\xi).
\]

We consider the fourth term of the left-hand side of \eqref{key}.
It is known that a Sasakian manifold is a K-contact manifold (cf. \cite{YK84}). Hence, we have
\[
\nabla _W\xi=-\phi W,
\]
for any vector field $W$ on $M$.
Therefore, the fourth term of the left-hand side of $\eqref{key}$ vanishes. In fact,
\begin{equation}\label{vanishei}
g(\nabla_{e_i}\xi,e_i)
=g(-\phi e_i,e_i)=0.
\end{equation}

We consider the fifth term of the left-hand side of $\eqref{key}$.

\[
\begin{aligned}
(e_i(\eta(\nabla F)))g(\phi X,e_i)
&= g(\phi X,\nabla (\eta(\nabla F)))\\
&= g(X,-\phi\nabla (\eta(\nabla F))).
\end{aligned}
\]

We consider the sixth term of the left-hand side of $\eqref{key}$.
By Proposition \ref{prop:Kcontact}, we have
\[
\nabla _{\phi X}\xi=-\phi(\phi X)=-\phi^2X=-(-X+\eta(X)\xi)=X-\eta(X)\xi.
\]
Hence, we have
\[
\begin{aligned}
g(e_i,\nabla F)g(\nabla _{\phi X}\xi,e_i)
&=
g(e_i,\nabla F)g(X-\eta(X)\xi,e_i)\\
&=
g(X-\eta(X)\xi,\nabla F)\\
&
=g(X,\nabla F-\eta(\nabla F)\xi).
\end{aligned}
\]
We consider the seventh term of the left-hand side of \eqref{key}.
\[
-(2n+1)(\phi(X))(\eta(\nabla F))
=
g(\phi X, -(2n+1)\nabla (\eta(\nabla F)))
=
g(X, (2n+1)\phi(\nabla (\eta(\nabla F)))).
\]
We consider the eighth term of the left-hand side of \eqref{key}.
\[
2ncg(\phi X,\nabla F)\eta(\nabla F)=-2ncg( X, \phi \nabla F)\eta(\nabla F)=g(X,-2nc\eta(\nabla F) \phi \nabla F).
\]
Combining these equations with \eqref{key}, we have
\begin{equation*}
\begin{aligned}
&g(X,-(2n-1)(\nabla \psi-c\psi\nabla F))
+g(X,-2n\eta(\nabla F)\xi-\eta(\nabla \psi-c\psi\nabla F)\xi)\\
&+g(X,-\phi\nabla (\eta(\nabla F)))
+g(X,\nabla F-\eta(\nabla F)\xi)\\
&+g(X, (2n+1)\phi(\nabla (\eta(\nabla F))))
+g(X,-2nc\eta(\nabla F) \phi \nabla F)=0.
\end{aligned}
\end{equation*}
Since $X$ is arbitrary, one has
\begin{equation}\label{key1}
\begin{aligned}
&-(2n-1)(\nabla \psi-c\psi\nabla F)
-2n\eta(\nabla F)\xi-\eta(\nabla \psi-c\psi\nabla F)\xi\\
&+2n\phi(\nabla (\eta(\nabla F)))
+\nabla F-\eta(\nabla F)\xi
-2nc\eta(\nabla F) \phi \nabla F
=0.
\end{aligned}
\end{equation}
We consider $\nabla (\eta(\nabla F))$.
\begin{equation}\label{nenf}
\begin{aligned}
\nabla (\eta(\nabla F))
&=e_ig(\nabla F,\xi)e_i\\
&=g(\nabla_{e_i}\nabla F,\xi)e_i+g(\nabla F,\nabla_{e_i}\xi)e_i\\
&=\psi \xi +c\eta(\nabla F)\nabla F+\phi\nabla F,
\end{aligned}
\end{equation}
where the last equality follows from \eqref{getm1} and Proposition \ref{prop:Kcontact}.
Hence, we have
\[
\phi(\nabla (\eta(\nabla F)))=c\eta(\nabla F)\phi\nabla F-\nabla F+\eta(\nabla F)\xi.
\]
Substituting this into \eqref{key1}, we have
\begin{equation}\label{rev1}
\begin{aligned}
-(2n-1)\{\nabla \psi-c\psi\nabla F+\nabla F\}
-\eta(\nabla \psi-c\psi\nabla F+\nabla F)\xi=0.
\end{aligned}
\end{equation}
Hence, one has
\[
\begin{aligned}
0&=
g(-(2n-1)\{\nabla \psi-c\psi\nabla F+\nabla F\}
-\eta(\nabla \psi-c\psi\nabla F+\nabla F)\xi,\xi)\\
&=-2n\eta(\nabla \psi-c\psi\nabla F+\nabla F),
\end{aligned}
\]
where we used \eqref{eq:acmetric}. Substituting this into \eqref{rev1}, we have
\begin{equation}\label{keykey}
\nabla \psi=(c\psi-1)\nabla F.
\end{equation}

If $c=0$, then \eqref{keykey} gives $\nabla(\psi+F)=0$. Hence, we have $\psi=k-F$ for some constant $k$.
Substituting this into \eqref{getm1}, we obtain $\nabla\nabla F=(k-F)g$.
Let $v:=F-k$. Then, we obtain 
\[
\nabla\nabla v=-vg.
\]

If $c\neq0$, set $u:=e^{-cF}$.
By \eqref{getm1}, we have
\[
\nabla_X\nabla u
=\nabla_X(-cu\nabla F)
=-cX(u)\nabla F-cu\nabla_X\nabla F
=-cu\psi X.
\]
Set $\tau:=-c\psi u$. By using \eqref{keykey} and $\nabla u=-cu\nabla F$, we have
\[
\nabla\tau
=-c(u\nabla\psi+\psi\nabla u)
=-c\left(u(c\psi-1)\nabla F-c\psi u\nabla F\right)
=cu\nabla F
=-\nabla u.
\]
Hence $\tau+u=k$ for some constant $k$, and therefore $\nabla\nabla u=(k-u)g$.
Let $v:=u-k$. Then, we obtain 
\[
\nabla\nabla v=-vg.
\]
Therefore, in both cases $c=0$ and $c\not=0$, we obtain 
\begin{equation}\label{obata}
\nabla\nabla v=-vg.
\end{equation}
By Theorem 2 in \cite{Tashiro65} and the proof of Theorem 5.7.5 in \cite{PetersenRG}, we have that $(M,g)$ is isometric to the unit sphere
$(I\times \mathbb{S}^{2n},dr^2+\sin^2r\,g_{\mathbb{S}^{2n}})$ where $I$ is a compact interval and $g_{\mathbb{S}^{2n}}$ is the Riemannian metric of the unit sphere $\mathbb{S}^{2n}$.
Furthermore, by the proof of Theorem 5.7.5 in \cite{PetersenRG}, we have
$v=a\cos r$. By the above argument, if $c=-\frac{\mu}{\beta}=0$ then we have 
\[
F(r)=k+a\cos r.
\]
If $c\neq0$ then we have 
\[
F(r)=-\frac{1}{c}\ln(k+a\cos r)
\qquad (k>a).
\]

\noindent
Case 2: $\beta=0$.
 
In this case, the equation for a gradient Einstein-type manifold is as follows.
\begin{equation}\label{beta}
\nabla F\otimes\nabla F=\psi g.
\end{equation}
By taking the trace, we have
\begin{equation}\label{betat}
|\nabla F|^2=(2n+1)\psi.
\end{equation}
Assume that $\nabla F(p)\not=0$ at some point $p\in M$.
Since $\dim M\geq2$, one can take a unit vector $e$ such that $g(\nabla F,e)=0$ at $p$. By \eqref{beta}, we have
\[
0=(g(\nabla F,e))^2=\psi(p).
\]
Hence, one has 
\[
\psi(p)=0.
\]
However, by \eqref{betat}, we have
\[
|\nabla F|^2(p)=(2n+1)\psi(p)=0,
\]
which is a contradiction.
Therefore, $F$ is constant and $M$ is trivial.
\end{proof}

In Case 2 of the proof of the theorem, we did not use the assumption that $M$ is a Sasakian manifold. Therefore, we also have the following.

\begin{corollary}
Any $m$-dimensional $(m\geq2)$ gradient Einstein-type manifold with $\alpha=\beta=0$ is trivial.
\end{corollary}


\noindent
\textbf{Acknowledgements.}~~
The author is partially supported by the Grant-in-Aid for Scientific Research (C), No.23K03107, Japan Society for the Promotion of Science.

\noindent
\textbf{Data availability statement.}~~Data sharing is not applicable to this article as no datasets were generated or analysed during the current study.

\noindent
{\bf Conflict of interest.}~~
There is no conflict of interest in the manuscript.

\bibliographystyle{amsbook}

\end{document}